\documentclass[a4paper]{amsart}

\usepackage{amssymb}

\usepackage[textsize=tiny,linecolor=gray,bordercolor=gray,backgroundcolor=white]{todonotes}

\usepackage[reftex]{theoremref}

\usepackage{eucal}

\usepackage{bbm}

\newtheorem{theorem}{Theorem}[section]
\newtheorem{proposition}[theorem]{Proposition}
\newtheorem{lemma}[theorem]{Lemma} 

\theoremstyle{definition} 
\newtheorem{definition}[theorem]{Definition}
\newtheorem{conventions}[theorem]{Conventions}

\theoremstyle{remark} 
 
\newtheorem{example}[theorem]{Example}

\newcommand{\set}[1]{\left\{#1\right\}}

\newcommand{\setP}[2]{\left\{\,#1\,\middle|\,#2\,\right\}}

\DeclareSymbolFont{sfoperators}{OT1}{cmss}{m}{n}
\DeclareSymbolFontAlphabet{\mathsf}{sfoperators}
\makeatletter\def\operator@font{\mathgroup\symsfoperators}\makeatother

\DeclareMathOperator{\add}{add} \DeclareMathOperator{\nil}{nil}

\DeclareMathOperator{\mmod}{mod}

\newcommand{\rep}[1]{%
  {%
    \small%
    \begin{matrix}%
      #1%
    \end{matrix}%
  }%
} \DeclareMathOperator{\Ext}{Ext} \DeclareMathOperator{\Hom}{Hom}
\DeclareMathOperator{\End}{End} \DeclareMathOperator{\gldim}{gldim}
\DeclareMathOperator{\domdim}{domdim} \DeclareMathOperator{\projdim}{projdim}
\DeclareMathOperator{\ind}{ind} 
 \DeclareMathOperator{\findim}{findim}
\DeclareMathOperator{\injdim}{injdim} \DeclareMathOperator{\Coker}{Coker}

 \newcommand{\op}{\mathrm{op}}

\usepackage{tikz-cd}

\begin{document}
  
\date{\today}
  
\title[Nakayama-type phenomena in higher Auslander--Reiten theory]{Nakayama-type
  phenomena in higher Auslander--Reiten theory}

\author[G.~Jasso]{Gustavo Jasso}
\address[Jasso]{Mathematisches Institut\\
  Universit\"at Bonn\\
  Endenicher Allee 60\\
  D-53115 Bonn\\
  GERMANY} \email{gustavo.jasso@@math.lu.se} \urladdr{https://gustavo.jasso.info}
\curraddr{Lund University, %
  Centre for Mathematical Sciences, %
  Sölvegatan 18A, %
  22100 Lund, %
  Sweden%
}

\author[J.~K\"ulshammer]{Julian K\"ulshammer}
\address[K\"ulshammer]{Institut f\"ur Algebra und Zahlentheorie\\
  Universit\"at Stuttgart\\
  Pfaffenwaldring 57\\
  70569 Stuttgart\\
  GERMANY} \email{julian.kuelshammer@math.uu.se}
\urladdr{https://www.katalog.uu.se/profile/?id=N18-1115}
\curraddr{Department of Mathematics University of Uppsala, Box 480, Uppsala,
  Sweden}

\begin{abstract}
  This paper surveys recent contructions in higher Auslander--Reiten theory. We
  focus on those which, due to their combinatorial properties, can be regarded
  as higher dimensional analogues of path algebras of linearly oriented type
  $\mathbb{A}$ quivers. These include higher dimensional analogues of Nakayama
  algebras, of the mesh category of type $\mathbb{Z}\mathbb{A}_\infty$ and the
  tubes, and of the triangulated category generated by an $m$-spherical object.
  For $m=2$, the latter category can be regarded as the higher cluster category
  of type $\mathbb{A}_\infty$ whose cluster-tilting combinatorics are controlled
  by the triangulations of the cylic apeirotope.
\end{abstract}

\thanks{Part of this research was carried out during a visit of the second
  author at MPIM Bonn. The authors would like to thank the anonymous referee for helpful comments.}

\maketitle

\section{Introduction}

The category of finite dimensional representations of the quiver
\begin{equation*}
  \mathbb{A}_\infty^\infty\colon \begin{tikzcd}\dots\arrow{r}
    &-1\arrow{r}&0\arrow{r}&1\arrow{r}&\dots\end{tikzcd}
\end{equation*}
is well understood. The indecomposable representations are the interval modules
and its Auslander--Reiten quiver is of type $\mathbb{Z}\mathbb{A}_\infty$. The
category of nilpotent finite dimensional representations of the cylic quiver
$\tilde{\mathbb{A}}_{n-1}$ with $n$ vertices, obtained as the quotient of
$\mathbb{A}_\infty^\infty$ by identifyng to vertices if they are equal modulo
$n$, has a similar description. In this case its Auslander--Reiten quiver is
rank $n$ tube. Note that both the mesh category of type $\mathbb{ZA}_\infty$ and
the tubes are frequently occurring shapes of Auslander--Reiten components. For
our purposes, it is important to note that the categories of representations of
the Nakayama algebras embedd into those of the above quivers and hence admit a
nice combinatorial description.

In this survey we introduce higher dimensional analogues of the quivers
$\mathbb{A}_\infty^\infty$ and $\tilde{\mathbb{A}}_{n-1}$ from the viewpoint of
Iyama's higher Auslander--Reiten theory \cite{Iya07b,Iya07}. This allows us to
construct higher dimensional analogues of the Nakayama algebras, whose higher
Auslander--Reiten theory is controlled by higher dimensional analogues of the
mesh category of type $\mathbb{ZA}_\infty$ and the tubes \cite{JK16}. To
complement loc. cit., we give a few examples to illustrate how the homological
behaviour of the higher Nakayama algebras differs from that of their classical versions.

We also construct certain Calabi--Yau triangulated categories which can be
regarded as higher dimensional analogues of the triangulated categories
generated by spherical objects of dimension greater than or equal to $2$. In
particular, in analogy to the case of a $2$-dimensional spherical object, we
obtain a higher version of the cluster category of type $\mathbb{A}_\infty$
introduced by Holm and J{\o}rgensen \cite{HJ12}. In this case, the
cluster-tilting combinatorics of this category is controlled by triangulations
of the cyclic apeirotope (a higher dimensional analog of the $\infty$-gon).

Let us briefly desribe the structure of this article. Section
\ref{sec:auslander-iyama} recalls the Auslander correspondence and Iyama's
generalisation of it obtained by introducing cluster-tilting subcategories. We
also include a short introduction to higher Auslander--Reiten theory. In Section
\ref{sec:higher_nakayama} we recall the construction of the higher Nakayama
algebras in terms of the higher dimensional analogues of the quivers
$\mathbb{A}_\infty^\infty$ and $\tilde{\mathbb{A}}_{n-1}$. In Section
\ref{sec:obstructions} we give examples of higher Nakayama algebras which have
different homological properties than those of their classical versions.
Finally, in Section \ref{sec:cluster} we introduce the higher dimensional
analogues of the triangulated categories generated by spherical objects and of
the cluster category of type $\mathbb{A}_\infty$.

\begin{conventions}
  We fix a positive integer $d$ as well as a field $\mathbbm{k}$. We denote the
  duality of the category of finite dimensional $\mathbbm{k}$-vector
  spaces by $D:=\Hom_{\mathbbm{k}}(-,\mathbbm{k})$. By algebra we mean
  associative unital $\mathbbm{k}$-algebra. All modules we consider are finite dimensional left modules.  	  Two algebras $\Lambda$ and
  $\Lambda'$ are \emph{Morita equivalent} if their categories of modules are
  equivalent. Let $\Lambda$ be a finite dimensional algebra. We denote by $\ind
  \Lambda$ a complete set of representatives of the isomorphisms classes of
  finite dimensional indecomposable $\Lambda$-modules. Let $M$ be a $\Lambda$-module (or, more generally, an object in some additive category). We denote
  by $\add M$ the full subcategory of $\mmod\Lambda$ consisting of all
  $\Lambda$-modules which are direct summands of $M^n$ for some positive integer
  $n$.
\end{conventions}

\section{Preliminaries}
\label{sec:auslander-iyama}

In this section we recall the basics of higher Auslander--Reiten theory, namely
Auslander--Iyama correspondence and the basic theory of
$d$-representation-finite $d$-hereditary algebras. The reader is referred to
\cite{Iya07, Iya11, Iya08, IO11, IJ16, JK16c} for further reading.

\subsection{The Auslander--Iyama correspondence}

For the purpose of motivation we first recall classical Auslander
correspondence. This establishes a relationship between
representation-finiteness, a \emph{representation-theoretic} property, and being
an Auslander algebra, a \emph{homological} property. Moreover, the minimal
projective resolutions of simple modules of projective dimension $2$ over
Auslander algebras provided the first evidence for the existence of almost-split
sequences. Thus, Auslander correspondence can also be considered as one of the
seminal results in Auslander--Reiten theory.

On the representation-theoretic side let $\Lambda$ be a finite dimensional
algebra. Recall that $\Lambda$ is \emph{representation-finite} if there are only
finitely many indecomposable $\Lambda$-modules up to isomorphism. An equivalent
way of saying this is that there exists a finite dimensional $\Lambda$-module
$M$ such that every indecomposable $\Lambda$-module is a direct summand of $M$.
Such a $\Lambda$-module $M$ is called a \emph{representation generator of
  $\mmod\Lambda$}. %Given a representation-finite algebra $\Lambda$, we call a pair $(\Lambda,M)$ with $M$ a representation generator of $\mmod \Lambda$ an \emph{Auslander pair} associated to $\Lambda$. %There is an equivalence relation on Auslander pairs induced by Morita equivalence, namely two
% Auslander pairs $(\Lambda,M)$ and $(\Lambda',M')$ are \emph{equivalent} if
% $\add M$ and $\add M'$ are equivalent categories.

On the homological side, let $\Gamma$ be a finite dimensional algebra and
\begin{equation*}
  \begin{tikzcd}[column sep=small]
    0\to\Gamma\to I^0\to I^1\to\cdots\to I^d\to\cdots
  \end{tikzcd}
\end{equation*}
a minimal injective coresolution of $\Gamma$. We remind the reader that the
\emph{dominant dimension of $\Gamma$} is defined as
\begin{equation*}
  \domdim\Gamma:=1+\sup\setP{i\in\mathbb{Z}}{I^i\text{ is projective}}.
\end{equation*}

A finite dimensional algebra $\Gamma$ is an \emph{Auslander algebra} if
\begin{equation*}
  \gldim\Gamma\leq 2\leq\domdim\Gamma.
\end{equation*}
Auslander algebras were introduced by the eponymous author in \cite{Aus71} where
the following theorem is also proven.

\begin{theorem}[Auslander correspondence]
  There is a one-to-one correspondence between Morita equivalence classes of
  representation-finite algebras and Morita equivalence classes of Auslander
  algebras. The correspondence associates to a re\-pre\-sen\-ta\-tion-finite
  algebra $\Lambda$ the algebra $\End_\Lambda(M)^{\op}$ for a representation
  generator $M$ of $\mmod \Lambda$. In the reverse direction, it associates to
  an Auslander algebra $\Gamma$ the algebra $\End_\Gamma(I^0)^{\op}$ where $I^0$
  is the injective hull of the regular representation.
\end{theorem}

Auslander correspondence is in fact a particular instance of Morita--Tachikawa
correspondence. This correspondence shifts the focus on the
`represen\-tation-theo\-retic' side away from the algebra $\Lambda$ to the
category $\add M$. In Auslander correspondence knowledge of $\Lambda$ is (up to
Morita equivalence) of course equivalent to knowledge of $\add M$ for $M$ a
representation generator.
 
Recall that a $\Lambda$-module $M$ is a \emph{generator-cogenerator of
  $\mmod\Lambda$} if every indecomposable projective and every indecomposable injective
$\Lambda$-module is a direct summand of $M$. In particular, a representation
generator is an example of a generator-cogenerator. A \emph{Morita--Tachikawa
  pair} is a pair $(\Lambda,M)$ consisting of a finite dimensional algebra
$\Lambda$ and generator-cogenerator $M$ of $\mmod\Lambda$. Two Morita--Tachikawa
pairs $(\Lambda,M)$ and $(\Lambda',M')$ are \emph{equivalent} if $\add M$ and
$\add M'$ are equivalent categories. Note that this implies that the algebras
$\Lambda$ and $\Lambda'$ are Morita equivalent: Indeed, a module $P$ is a projective object in $\add M$ (in the sense that every epimorphism to it splits) if and only if it is a projective $\Lambda$-module. It is easy to see that no other module can be projective as $M$ is a generator. For the converse let $P$ be a projective $\Lambda$-module and $f\colon M\to P$ be an epimorphism. Let $g$ be the composition of the projection $P\to \Coker f$ and the inclusion $\Coker f\to I$ into its injective hull. Then, $gf=0$ implies that $g=0$ as $f$ is an epimorphism. In particular $\Coker f=0$. It follows that $\add \Lambda\cong \add \Lambda'$ as both are defined by categorical properties in $\add M\cong \add M'$. Thus, $\Lambda$ and $\Lambda'$ are Morita equivalent.

\begin{theorem}[Morita--Tachikawa correspondence]
  There is a one-to-one correspondence between equivalence classes of
  Morita--Tachikawa pairs and Morita equivalence classes of finite dimensional
  algebras $\Gamma$ such that $\domdim\Gamma\geq 2$. The correspondence
  associates to a Morita--Tachikawa pair $(\Lambda,M)$ the algebra
  $\End_\Lambda(M)^{\op}$ and, in the other direction, associates to an algebra
  $\Gamma$ the Morita--Tachikawa pair $(\End_{\Gamma}(I^0)^{\op},I^0)$ for $I^0$
  the injective hull of $\Gamma$ as a left $\Gamma$-module.
\end{theorem}

Almost thirty years later, Iyama realised that Auslander correspondence,
together with it the most fundamental aspects of Auslander--Reiten theory, could
be vastly generalised. In some sense, Iyama's approach was inverse to that of
Auslander. Indeed, the class of Auslander algebras admits an ``obvious''
generalisation: a finite dimensional algebra $\Gamma$ is a \emph{$d$-Auslander
  algebra} if the inequalities
\begin{equation*}
  \gldim\Gamma\leq d+1\leq\domdim\Gamma
\end{equation*}
are satisfied.

In view of the Morita--Tachikawa correspondence, the modules which are in
correspondence with $d$-Auslander algebras must be particular kinds of
generator-cogenerators. These are the $d$-cluster-tilting modules, which we now
define.

\begin{definition}[{\cite{Iya07, Iya11, IO11}}]
  Let $\Lambda$ be a finite dimensional algebra. A $\Lambda$-module $M$ is
  \emph{$d$-cluster-tilting} if
  \begin{align*}
    \add M&=\setP{X\in\mmod\Lambda}{\forall i\in\set{1,\dots,d-1}:\Ext^i_\Lambda(X,M)=0}\\
          &=\setP{Y\in\mmod\Lambda}{\forall i\in\set{1,\dots,d-1}:\Ext^i_\Lambda(M,Y)=0}.
  \end{align*}
  An algebra $\Lambda$ is called \emph{weakly $d$-representation-finite} if
  there exists a $d$-cluster-tilting module.
\end{definition}

It is easy to see that a $1$-cluster-tilting $\Lambda$-module is precisely a
representation generator of $\mmod\Lambda$, and that every $d$-cluster-tilting
$\Lambda$-module is a generator-cogenerator.

A \emph{$d$-Auslander--Iyama pair} is a pair $(\Lambda,M)$ where $\Lambda$ is a
finite dimensional algebra and $M$ is a $d$-cluster-tilting module. Two
$d$-Auslander--Iyama pairs $(\Lambda,M)$ and $(\Lambda',M')$ are
\emph{equivalent} if $\add M$ and $\add M'$ are equivalent. The following
theorem appeared originally as \cite[Theorem 0.2]{Iya07}.

\begin{theorem}[$d$-dimensional Auslander--Iyama correspondence]
  There is a one-to-one correspondence between equivalence classes of
  $d$-Aus\-lan\-der--Iyama pairs and Morita equivalence classes of $d$-Auslander
  algebras. The correspondence associates to a $d$-Auslander--Iyama pair
  $(\Lambda,M)$ the $d$-Auslander algebra $\End_\Lambda(M)^{\op}$, and in the
  inverse direction to a $d$-Auslander algebra $\Gamma$ the pair
  $(\End_\Gamma(I^0)^{\op},I^0)$ where $I^0$ is the injective hull of the
  regular representation of $\Gamma$.
\end{theorem}

Recently, there have been further generalisations of this correspondence (and
thus instances of Morita--Tachikawa correspondence) by Iyama--Solberg
\cite{IS16} and Marczinzik \cite{Mar17} replacing $\gldim \Gamma$ by
$\injdim_\Gamma \Gamma$ and $\findim \Gamma$, respectively.

\subsection{Higher Auslander--Reiten theory}

Following the steps in the development of Auslander--Reiten theory, the
$d$-di\-men\-sio\-nal Auslander--Iyama correspondence served as the starting
point for a $(d+1)$-dimensional Auslander--Reiten theory. The minimal projective
resolutions of simple modules of projective dimension $d+1$ over $d$-Auslander
algebras give rise to the prototypical $d$-almost-split sequences. The only
caveat is that, in order to obtain analogous properties to those in classical
Auslander--Reiten theory, one needs to restrict to modules in a
$d$-cluster-tilting subcategory, a generalisation of the category $\add M$ for a
$d$-cluster-tilting module $M$.

\begin{definition}[{{\cite[(2.4)]{Iya07}, \cite[Definition 2.2]{Iya07b}},
  \cite[Definition 3.14]{Jas16}}]
  \label{def:cluster-tilting}
  Let
  $\mathcal{A}$ be an abelian category. A generating-cogenerating functorially
  finite (full) subcategory $\mathcal{C}$ of $\mathcal{A}$ is a
  \emph{$d$-cluster-tilting subcategory}\footnote{Note that $d$-cluster-tilting
    subcategories where originally called ``maximal $(d-1)$-orthogonal
    subcategories'' in \cite{Iya07}.} if
  \begin{align*}
    \mathcal{C}&=\setP{X\in \mathcal{A}}{\forall i\in\set{1,\dots,d-1}:\Ext^i_\mathcal{A}(X,\mathcal{C})=0}\\
               &=\setP{Y\in \mathcal{A}}{\forall i\in\set{1,\dots,d-1}:\Ext^i_{\mathcal{A}}(\mathcal{C},Y)=0}.
  \end{align*}
  An object $C\in \mathcal{A}$ is a \emph{$d$-cluster-tilting object} if $\add
  C$ is a $d$-cluster-tilting subcategory of $\mathcal{A}$.
\end{definition}

The authors are not aware of any examples of abelian categories $\mathcal{A}$ having a subcategory $\mathcal{C}$ which satisfies the assumptions of a $d$-cluster-tilting subcategory except for being generating-cogenerating. However, this assumption is needed in the proofs of even the basic statements of the theory in \cite{Jas16}. If $\mathcal{A}$ has enough projectives and enough injectives, this condition is automatic as $\mathcal{C}$ always contains all the projective and all the injective objects of $\mathcal{A}$.

Let $\mathcal{C}$ be a $d$-cluster-tilting subcategory of an abelian category
$\mathcal{A}$. Then, a \emph{$d$-almost split sequence} is an exact sequence
\[0\to C_0\to C_1\to \dots\to C_d\to C_{d+1}\to 0\] with $C_i\in \mathcal{C}$
for all $i$ and $C_0$ indecomposable such that every non-split epimorphism
$C'\to C_{d+1}$ for some $C'\in \mathcal{C}$ factors through the map $C_d\to
C_{d+1}$; moreover, one requires that for each $i\in\set{0,1,\dots,d}$ the
morphisms $C_i\to C_{i+1}$ lies in the Jacobson radical of $\mmod\Lambda$. In
this case, $C_0\cong \tau_d(C_{d+1})$ where
\begin{equation*}
  \tau_d(M):=\tau\Omega^{d-1}(M)
\end{equation*}
is the $d$-dimensional replacement of the classical Auslander--Reiten
translation $\tau$. Here, $\Omega$ denotes Heller's syzygy functor.

In `classical' representation theory, representation-finite hereditary algebras
were among the first algebras whose Auslander--Reiten theory was systematically
investigated, leading to fruitful results, see for example
\cite{ASS06,SS07a,SS07b}. Modelled on this, weakly $d$-representation finite
algebras of small global dimension were among the first to be investigated.

\begin{definition}
  A finite dimensional algebra $\Lambda$ is \emph{$d$-representation-finite
    $d$-hereditary} if it is weakly $d$-representation-finite and $\gldim
  \Lambda\leq d$.
\end{definition}

As an exercise, the reader can verify that $d$ is the smallest possible global
dimension for a non-semisimple weakly $d$-representation-finite algebra. Thus, a
$d$-representation-finite $d$-hereditary algebra has global dimension either $0$
or $d$. Much of the research in higher Auslander--Reiten theory has focused on
the class of $d$-representation-finite $d$-hereditary algebras, see for example
\cite{IO11,HI11,HI11b, OT12, IO13}. In the beginning of Section
\ref{sec:higher_nakayama} we recall Iyama's construction of the higher Auslander
algebras of type $\mathbb{A}$, which are a particularly important class of
$d$-representation-finite $d$-hereditary algebras.

There is an analogous result to the fact that for a representation-finite
hereditary algebra all indecomposable modules can be constructed from the
indecomposable injectives by iterative application of the Auslander--Reiten
translation.

\begin{theorem}[{\cite[Proposition 1.3(b)]{Iya11}}]
  Let $\Lambda$ be a $d$-representation-finite $d$-hereditary algebra with $n$ simple modules $S_1,\dots,S_n$ up to isomorphism. Then, the
  $\Lambda$-module
  \begin{equation*}
    M=\bigoplus_{i=1}^n\bigoplus_{j\geq 0} \tau_d^j I_i,
  \end{equation*}
  where $I_i$ denotes the indecomposable injective $\Lambda$-module with socle
  $S_i$, is the (up to isomorphism) unique basic $d$-cluster-tilting
  $\Lambda$-module.
\end{theorem}

Another nice feature of $d$-representation-finite $d$-hereditary algebras is
that their derived categories also have a $d$-cluster-tilting subcategory. Note that $d$-cluster-tilting subcategories of triangulated categories
can be defined in an obvious way (omitting the condition of being
generating-cogenerating) with the obvious generalisation of almost split
triangles to almost split $d$-angles.

For a finite dimensional algebra $\Lambda$, set
\begin{equation*}
  \mathcal{U}(\Lambda):=\add\setP{\nu_d^i \Lambda}{i\in \mathbb{Z}},
\end{equation*}
where $\nu=D\Lambda\otimes_\Lambda^{\mathbb{L}}-$ is the derived Nakayama
functor and $\nu_d:=\nu[-d]$. Thus, $\mathcal{U}(\Lambda)$ is a full subcategory
of the bounded derived category $D^b(\mmod \Lambda)$.
  
\begin{theorem}
  Let $\Lambda$ be a $d$-representation-finite $d$-hereditary algebra. Then,
  $\mathcal{U}(\Lambda)$ is a $d$-cluster-tilting subcategory of $D^b(\mmod
  \Lambda)$.
\end{theorem}

Note that an analogous construction for more general weakly
$d$-representation-finite algebra is not known and a naive extension is even not
possible, see the authors' note \cite{JK16b}.

\section{Higher Nakayama algebras}\label{sec:higher_nakayama}

In this section we recall the construction of higher Nakayama algebras from the
authors' article \cite{JK16}. For this, we introduce a certain universal
category whose importance in higher Auslander--Reiten theory is already
mentioned by Iyama in \cite[Introduction]{Iya11}.

Every partially ordered $(P,\leq)$ has an associated $\mathbbm{k}$-category
$A(P)$ called the \emph{incidence $\mathbbm{k}$-category of the poset $P$},
whose objects are the elements of $P$ and with morphisms
\begin{equation*}
  A(P)(x,y):=\begin{cases}
    \mathbbm{k}f_{yx}&\text{if } x\leq y,\\
    0&\text{otherwise},
  \end{cases}
\end{equation*}
where $f_{yx}$ denotes a basis vector of a corresponding one-dimensional space.
The composition in $A(P)$ induced by $f_{zy}\circ f_{yx}:=f_{zx}$. If the set
$P$ is finite, this is just the incidence algebra of the poset $(P,\leq)$.

\begin{definition}
  Let $\leq$ be the product order on $\mathbb{Z}^d$, i.e. $\alpha\leq \beta$ if
  and only if $\alpha_i\leq \beta_i$ for all $i\in \set{1,\dots,d}$. Let
  $A^{(d)}:=A(\mathbb{Z}^d,\leq)$.
\end{definition}

The $\mathbbm{k}$-category $A^{(d)}$ can be explicitly described as follows. Let
$Q^{(d)}$ be the quiver with vertex set $\mathbb{Z}^d$ and arrows
\begin{equation*}
  a_i=a_i(\alpha)\colon \alpha\to\alpha+e_i
\end{equation*}
where $i\in\set{1,\dots,d}$. Let $I$ be the two-sided ideal of the path category
$\mathbbm{k}Q^{(d)}$ of $Q^{(d)}$ generated by all commutativity relations
\begin{equation*}
  \setP{a_ia_j-a_ja_i}{1\leq i<j\leq d}.
\end{equation*}
With these conventions, $A^{(d)}$ is isomorphic to the quotient category
$\mathbbm{k}Q^{(d)}/I$. From both descriptions it is clear that $A^{(1)}$ is
isomorphic to the path category of $\mathbb{A}_\infty^\infty$ with linear
orientation.

Note that $A^{(d)}$ is naturally a graded $\mathbbm{k}$-category with respect to
path length. It is elementary to verify that $A^{(d)}$ is the universal covering
of the polynomial ring in $d$ variables in the sense of \cite{Gre83} (see also \cite{MVdlP83}), that it is of global dimension $d$ and
a Koszul category in the sense of \cite{MOS09}.

In the sequel we heavily use a certain automorphism $\varphi_d$ on $A^{(d)}$
(and idempotent quotients thereof) defined by
\begin{align*}
  \varphi_d(\mu_1,\dots,\mu_d)=(\mu_1-1,\dots,\mu_d-1),\\
  \varphi_d(a_i(\mu_1,\dots,\mu_d))=a_i(\mu_1-1,\dots,\mu_d-1).
\end{align*}

\subsection{The universal $d$-Nakayama category}

In this subsection, we introduce another category, a certain idempotent quotient of $A^{(d)}$, which serves as the universal higher Nakayama category, much like the path
category of $\mathbb{A}_\infty^\infty$ serves as the universal Nakayama
category. Set
\[\mathbf{os}^{(d)}:=\setP{(\mu_1,\dots,\mu_d)\in \mathbb{Z}^{d}}{\mu_1\geq
    \mu_2\geq \dots\geq \mu_d}.\]

\begin{definition}
  For $d\geq 2$, the idempotent quotient $A^{(d)}_\infty$ of $A^{(d)}$ by the
  objects not in $\mathbf{os}^{(d)}$ is called the \emph{universal $d$-Nakayama
    category}.
\end{definition}

Note that $A_\infty^{(d)}$ is different from the incidence category of
$\mathbf{os}^{(d)}$ as there are certain zero relations at the `boundary' of
$A^{(d)}_\infty$. Furthermore observe that $A_\infty^{(2)}$ is the mesh category
of a $\mathbb{Z}\mathbb{A}_\infty$-component, one of the most typical
Auslander--Reiten components in `classical' Auslander--Reiten theory. This is
why we chose to give them the similar name $A_\infty^{(d)}$ as we see them as
higher analogues of this category.

\begin{definition}
  Let $\lambda\in\mathbf{os}^{(d+1)}$. The \emph{interval module} $M_{\lambda}$
  is the $A_\infty^{(d)}$-module with composition factors
  \[[M_{\lambda}\colon S_\kappa]=\begin{cases}\mathbbm{k}&\text{if $\kappa\in
        [(\lambda_2,\dots,\lambda_{d+1}),(\lambda_1,\dots,\lambda_{d})]$,}\\0&\text{else,}\end{cases}\]
  where $[(\lambda_2,\dots,\lambda_{d+1}),(\lambda_1,\dots,\lambda_{d})]$
  denotes the interval in the partial order on $\mathbf{os}^{(d)}$. The maps on
  the arrows $a_i$ are identity maps between different copies of $\mathbbm{k}$
  and $0$ elsewhere.
\end{definition}

The category $A_\infty^{(d)}$ satisfies analogous properties to being
$d$-representation-finite in the infinite setting:

\begin{theorem}[{\cite[Theorem 2.3.5]{JK16}}]
  The abelian category $\mmod A_\infty^{(d)}$ of finite dimensional
  $A_\infty^{(d)}$-modules is of global dimension $d$ and  has a $d$-cluster-tilting subcategory consisting of
  the additive hull $\mathcal{M}_\infty^{(d)}$ of all interval modules
  $M_{\lambda}$ for $\lambda\in \mathbf{os}^{(d+1)}$.

  In the derived situation,
  \[\mathcal{U}_\infty^{(d)}:=\add \setP{M_\lambda[di]}{\lambda\in
      \mathbf{os}^{(d+1)}, i\in \mathbb{Z}}\] is a $d$-cluster-tilting
  subcategory of $D^b(\mmod A_\infty^{(d)})$.
\end{theorem}

\subsection{Higher Nakayama algebras}

In this section, we survey on the results of the authors' paper \cite{JK16} on
the construction of higher analogues of Nakayama algebras.

\begin{definition}
  A finite dimensional algebra $\Lambda$ is called a \emph{Nakayama algebra} (or
  \emph{uniserial}) if every indecomposable projective and every indecomposable
  injective module is uniserial.
\end{definition}

It is well-known that (over an algebraically closed field), the connected
Nakayama algebras are precisely the quotients by admissible ideals of the path
algebras of the linearly oriented Dynkin quiver
\begin{equation*}
  \mathbb{A}_n\colon\begin{tikzcd}
    0\arrow{r}&1\arrow{r}&2\arrow{r}&\cdots\arrow{r}&n-1,
  \end{tikzcd}
\end{equation*}
and the linearly oriented extended Dynkin quiver
\begin{equation*}
  \begin{tikzpicture}
    \node (LABEL) at (180:3) {$\widetilde{\mathbb{A}}_{n-1}\colon$};
    \node[fill=white] (0) at (180-90:1.4) {$0$}; \node[fill=white] (1) at
    (120-90:1.4) {$1$}; \node[fill=white] (2) at (60-90:1.4) {$2$};
    \node[fill=white] (3) at (0-90:1.4) {$\cdots$}; \node[fill=white] (4) at
    (300-90:1.4) {$n-2$}; \node[fill=white] (5) at (240-90:1.4) {$n-1$}; \draw
    [->] (0) edge (1) (1) edge (2) (2) edge (3) (3) edge (4) (4) edge (5) (5)
    edge (0);
  \end{tikzpicture}
\end{equation*}

Note that the path algebra of $\mathbb{A}_n$ is an idempotent truncation of
$A_\infty^{(1)}$ while the path algebra of $\tilde{\mathbb{A}}_{n-1}$ is
obtained from $A_\infty^{(1)}$ by dividing out by the action of $\varphi_1^n$.
There are analogues of the path algebras of $\mathbb{A}_n$ and
$\tilde{\mathbb{A}}_{n-1}$, respectively. The former was already introduced in
\cite{Iya11} and studied extensively in \cite{OT12}. For $n=1$, the latter can
be seen as a higher preprojective algebra of type $\mathbb{A}_\infty$. We start
by introducing the higher analogues of the path algebra of $\mathbb{A}_n$, in
\cite{IO11} called higher Auslander algebras of type $\mathbb{A}$.

\begin{definition}
  Let $\mathbf{os}^{(d)}_n:=\setP{(\mu_1,\dots,\mu_d)\in
    \mathbf{os}^{(d)}}{n-1\geq \mu_1, \mu_d\geq 0}$. The \emph{$d$-Auslander
    algebra $A_n^{(d)}$ of type $\mathbb{A}_n$} is defined to be the (full)
  subcategory of $A_\infty^{(d)}$ on the vertices in $\mathbf{os}^{(d)}_n$.
\end{definition}

\begin{theorem}[{\cite[Theorem 1.18]{Iya11}}]\label{thm:higher_auslander_type_A}
  Let $d\geq1$. Then, $A_n^{(d)}$ is a $d$-representation-finite $d$-hereditary
  algebra. Moreover, $A_n^{(d+1)}\cong \End_{A_n^{(d)}}(M_n^{(d)})$ where
  $M_n^{(d)}$ is the unique basic $d$-cluster-tilting $A_n^{(d)}$-module. The
  indecomposable direct summands of $M_n^{(d)}$ are precisely the interval
  modules $M_{\lambda}$ for $\lambda\in \mathbf{os}^{(d+1)}_n$.
\end{theorem}

The following theorem of Iyama characterises the higher Auslander algebras of
type $\mathbb{A}$ as the only finite dimensional algebras which are
$d$-representation finite $d$-hereditary and $(d-1)$-Auslander algebras at the
same time. Thus, in particular, these are the only $d$-representation-finite
$d$-hereditary algebra which can be constructed inductively as in Theorem
\ref{thm:higher_auslander_type_A}.

\begin{theorem}[{\cite[Theorem 1.19]{Iya11}}]
  Let $d\geq2$ and $\Lambda$ a finite dimensional algebra. Then, $\Lambda$ is
  $d$-representation-finite $d$-hereditary as well as a $(d-1)$-Auslander
  algebra if and only if it is Morita equivalent to $A_n^{(d)}$ where $n$ is the
  number of simple $\Lambda$-modules in the $d$-cluster-tilting subcategory of
  $\mmod \Lambda$.
\end{theorem}

For constructing the higher analogue of the path algebra of $\tilde{\mathbb{A}}$
note that $A_\infty^{(d)}$ is closed under $\varphi_d$. Recall that $\varphi_d$ was defined to be the quiver automorphism satisfying
\begin{align*}
  \varphi_d(\mu_1,\dots,\mu_d)=(\mu_1-1,\dots,\mu_d-1),\\
  \varphi_d(a_i(\mu_1,\dots,\mu_d))=a_i(\mu_1-1,\dots,\mu_d-1).
\end{align*}
Define the higher
analogue of the path algebra of $\tilde{\mathbb{A}}_{n-1}$ by
$\tilde{A}_{n-1}^{(d)}:=A_\infty^{(d)}/\varphi_d^n$. It is easy to see that the
images of the $A_\infty^{(d)}$-modules $M_{\lambda}$ and $M_{\mu}$ under the
pushdown functor are isomorphic if and only if $\lambda\equiv \mu\mod n\mathbb{Z}(1,1,\dots,1)$.
As the quiver of this category contains oriented cycles, we consider the
category of finite dimensional nilpotent modules instead of the category of
finite dimensional modules.

\begin{theorem}[{\cite[Theorem 2.3.4]{JK16}}]
  The category $\nil \tilde{A}^{(d)}_{n-1}$ has a $d$-cluster-tilting
  subcategory $\mathcal{T}^{(d)}_n$ whose objects are given by the images of the
  interval modules $M_{\lambda}$ where $\lambda\in \mathbf{os}^{(d+1)}_n$. Moreover, $\mathcal{T}^{(d)}_n\cong \add \tilde{A}_{n-1}^{(d+1)}$.
\end{theorem}

The name $\mathcal{T}_n^{(d)}$ is chosen because for $d=1$,
$\mathcal{T}_n^{(d)}$ is the tube of rank $n$ from `classical' Auslander--Reiten
theory.

After having defined the analogues of the path algebras of $\mathbb{A}_n$ and
$\tilde{\mathbb{A}}_{n-1}$, the next step is to find analogues for admissible
relations for Nakayama algebras. In contrast to classical Nakayama algebras,
which arise as quotients by admissible ideals of the algebras $A_n^{(1)}$ and
$\tilde{A}_{n-1}^{(1)}$, the higher Nakayama algebras arise as idempotent
quotients from $A_n^{(d)}$ and $\tilde{A}_{n-1}^{(d)}$, respectively.

Recall that a convenient combinatorial way to label Nakayama algebras is by means of the vector
\begin{equation*}
(\dim P_0,\dim P_1,\dots,\dim P_{n-1})
\end{equation*}
where $P_i$ is the indecomposable projective to the vertex
$i$. This vector is called the \emph{Kupisch series} corresponding to the
Nakayama algebra. Given a Kupisch series
$\underline{\ell}=(\ell_0,\dots,\ell_{n-1})$, the corresponding set of relations
one needs to divide out by is given by the paths $\setP{i-\ell_i\leadsto
  i}{\ell_{i-1}\geq \ell_i}$. Note that the interval modules $M_{(\mu_1,\mu_2)}$
for the Nakayama algebra with Kupisch series $\underline{\ell}$ satisfy
$\mu_1-\mu_2+1\leq \ell_{\mu_1}$ as $\mu_1-\mu_2+1$ gives the length of the
interval module $M_{(\mu_1,\mu_2)}$ and $\mu_1$ gives its top and every module
is a quotient of the corresponding indecomposable projective module with the
same top.

Analogously, in the higher setting for $d\geq 2$, the higher Nakayama algebra
$A_{\underline{\ell}}^{(d)}$ (of type $\mathbb{A}$, i.e. for $\ell_0=1$)
associated to the Kupisch series $\underline{\ell}$ is defined to be the
idempotent quotient of $A_n^{(d)}$ by the idempotents not in
\[\mathbf{os}^{(d)}_{\underline{\ell}}:=\setP{(\mu_1,\dots,\mu_d)\in
    \mathbf{os}_n^{(d)}}{\mu_1-\mu_d+1\leq \ell_{\mu_1}}.\] For the Nakayama
algebras of type $\tilde{\mathbb{A}}_{n-1}$ given the Kupisch series $\ell$ one
defines an infinite series $\tilde{\ell}$ by periodically extending $\ell$. One
then defines
\[\mathbf{os}^{(d)}_{\tilde{\underline{\ell}}}:=\setP{(\mu_1,\dots,\mu_d)\in
    \mathbf{os}^{(d)}}{\mu_1-\mu_d+1\leq \tilde{\ell}_{\mu_1}}.\] Note that by
construction, $\mathbf{os}^{(d)}_{\tilde{\underline{\ell}}}$ is stable under
$\varphi_d^n$. Define $\tilde{A}_{\underline{\ell}}^{(d)}$ to be the idempotent
quotient of $\tilde{A}_{n-1}^{(d)}$ by the idempotents not corresponding to the
ones in $\mathbf{os}^{(d)}_{\tilde{\underline{\ell}}}$.

% an be specified by a set of intervals
% $R=\setP{([\alpha,\beta],r)}{\alpha,\beta\in \mathbb{Z}_n, r\in \mathbb{N}_0}$
% where $\mathbb{Z}_n$ is endowed with a cyclic ordering and the admissible set
% of paths which are quotient out by are the paths from $\alpha$ to $\beta$ with
% winding number $r$, i.e. they start in $\alpha$, end in $\beta$ and have
% length $\beta-\alpha+nr$. Let
% $\tilde{R}:=\setP{[\alpha+mn,\beta+(r+m)n]\subseteq
% \mathbb{Z}}{([\alpha,\beta],r)\in R, n\in \mathbb{Z}}$. Define
% $\tilde{\mathbf{os}}_R^d:= \setP{(\mu_1,\dots,\mu_d)\in
% \mathbf{os}_\infty^{(d)}}{[\alpha,\beta]\nsubseteq [\mu_1,\mu_d] \text{ for
% all } [\alpha,\beta]\in \tilde{R}}$\todo{please check the combinatorics}.
% Define $A_\infty^{(d)}/\tilde{R}$ as the quotient of $A_\infty^{(d)}$ by the
% objects not in $\mathbf{os}_R^{(d)}$. Note that $A_\infty^{(d)}/\tilde{R}$ is
% stable under $\varphi_d$. Define
% $\tilde{A}_{n-1}^{(d)}/R:=(A_\infty^{(d)}/\tilde{R})/\varphi_d^n$.

\begin{theorem}
  For every Kupisch series $\underline{\ell}$ and every $d\geq 1$, the algebra
  $A_{\underline{\ell}}^{(d)}$ (respectively
  $\tilde{A}_{\underline{\ell}}^{(d)}$ in type $\tilde{\mathbb{A}}$) is weakly
  $d$-representation-finite. The modules in the corresponding
  $d$-cluster-tilting module are the interval modules $M_{\lambda}$ with
  $\lambda\in \mathbf{os}^{(d+1)}_{\underline{\ell}}$ (respectively $\lambda\in
  \mathbf{os}^{(d+1)}_{\tilde{\underline{\ell}}}$ in type $\tilde{\mathbb{A}}$).
\end{theorem}

Among all weakly $d$-representation-finite algebras, the higher Nakayama
algebras belong to a more special class of algebras, since their cluster-tilting
modules $M$ satisfy the stronger condtion
\[\Ext^i(M,M)\neq 0\implies i\in d\mathbb{Z}.\]
One of their special features is that
exact sequences of length $d$ in the cluster-tilting subcategory give rise to
long exact $\Ext^d$-sequences, see \cite{IJ16}. Such algebras are called
\emph{$d\mathbb{Z}$-representation-finite}. They can also be characterised by
the fact that the $d$-cluster-tilting module is closed under $\Omega^d$.

\section{Obstructions to an alternative definition of higher Nakayama algebras}
\label{sec:obstructions}

The reader may ask whether there is a characterisation of higher Nakayama
algebras similar to projectives and injectives being uniserial. The main problem
in answering this question is that it is not clear what the analogue of a
filtration is when there are only higher $\Ext$-groups, but no $\Ext^1$. In this
section we discuss a few more problems with such definition and with statements
involving the simples in a cluster-tilting subcategory.

\subsection{Simple modules and the $\Ext^d$-quiver of the higher Nakayama
  algebras}

We start by characterising simples only in terms of the cluster tilting
subcategory.

\begin{definition}
  Let $\mathcal{A}$ be an additive category and $S\in\mathcal{A}$ a non-zero
  object. We say that $S$ is \emph{Schur simple} if every non-zero morphism
  $X\to S$ is an epimorphism and every non-zero morphism $S\to Y$ is a
  monomorphism.
\end{definition}

It turns out that, with this definition, Schur simples of a cluster tilting
subcategory are precisely the simple modules.

\begin{lemma}
  Let $\mathcal{A}$ be an abelian category and $\mathcal{M}\subseteq\mathcal{A}$
  a functorially finite generating-cogenerating subcategory. Then, an object
  $S\in\mathcal{M}$ is Schur simple if and only if it is a simple object in
  $\mathcal{A}$.
\end{lemma}

\begin{proof}
  Let $S\in \mathcal{M}$ be simple in $\mathcal{A}$. Then, of course, it is
  Schur simple. For the other direction we show that a Schur simple object in
  $\mathcal{M}$ is in fact Schur simple in $\mathcal{A}$, and hence simple. For
  this, let $S\in \mathcal{M}$ be Schur simple and $ X\in \mathcal{A}$ be
  arbitrary with a morphism $f\colon S\to X$. Since $\mathcal{M}$ is covariantly
  finite, there exists a left $\mathcal{M}$-approximation $g\colon X\to M$ for
  some $M\in \mathcal{M}$. Since $S$ is Schur simple in $\mathcal{M}$, the
  composition $gf$ is a monomorphism or zero. Since the left approximation $g$
  is a monomorphism since $\mathcal{M}$ is cogenerating, it follows that $f$ is
  a monomorphism or zero. Dually, one shows that every non-zero morphism $X\to
  S$ in $\mathcal{A}$ is an epimorphism. It follows that $S$ is Schur simple in
  $\mathcal{A}$ which is equivalent to being simple.
\end{proof}

\begin{definition}
  \th\label{extd-quiver} Let $\mathcal{A}$ be an abelian category and
  $\mathcal{M}\subseteq\mathcal{A}$ a $d$-cluster-tilting subcategory. The
  \emph{$\Ext^d$-quiver} of $\mathcal{M}$ is the quiver which has as vertices
  the isomorphism classes of Schur simples in $\mathcal{M}$ and the number of
  arrows $[S]\to [S']$ is given by $\dim_\mathbbm{k}\Ext^d_\mathcal{A}(S,S')$.
\end{definition}
 
The $\Ext^d$-quivers of the higher Nakayama algebras are precisely what one
would expect, namely equal to the $\Ext^1$-quivers of the corresponding
(classical) Nakayama algebras.

\begin{proposition}
  Let $\underline{\ell}$ be a Kupisch series. Then, the the $\Ext^d$-quiver of
  $A_{\underline{\ell}}^{(d)}$ is $\mathbb{A}_n$ in case $\ell_0=1$. while the
  $\Ext^d$-quiver of $\tilde{A}_{\underline{\ell}}^{(d)}$ is
  $\widetilde{\mathbb{A}}_{n-1}$ in case $\ell_0\neq 1$.
\end{proposition}

\begin{proof}
  This follows immediately from the results of \cite{JK16}.
\end{proof}

We now give an example showing that, in general the $\Ext^d$-quiver for
$d$-cluster-tilting subcategories of module categories of finite dimensional
algebras is ill-behaved.
 
 \begin{example}
   \th\label{bad-ext2} Let $A$ be the preprojective algebra of type
   $\mathbb{A}_3$ which is also the higher Nakayama algebra
   $\tilde{A}_{(3)}^{(2)}$, that is, $A$ is the preprojective algebra of
   $\mathbb{A}_3$, i.e. the path algebra of the quiver
   \begin{center}
     \includegraphics{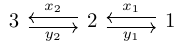}
   \end{center}
   bound by the relations $y_1x_1=0$, $x_2y_2=0$ and $x_1y_1=y_2x_2$. The reader
   can verify that
   \[
     M=A\oplus\rep{1}\oplus\rep{2\\1}\oplus\rep{1\\2}\quad\text{and}\quad
     N=A\oplus\rep{2\\1}\oplus\rep{2\\3}\oplus\rep{2\\13}
   \]
   are $2$-cluster-tilting $A$-modules, $M$ being the one given by interval
   modules. For the convenience of the reader, the Auslander--Reiten quiver of
   $\mmod A$ is given below, where indecomposable $A$-modules are described by
   their radical filtration and the indecomposable direct summands of $M$ and
   $N$ are indicated with rectangles and circles respectively.
   \begin{center}
     \includegraphics{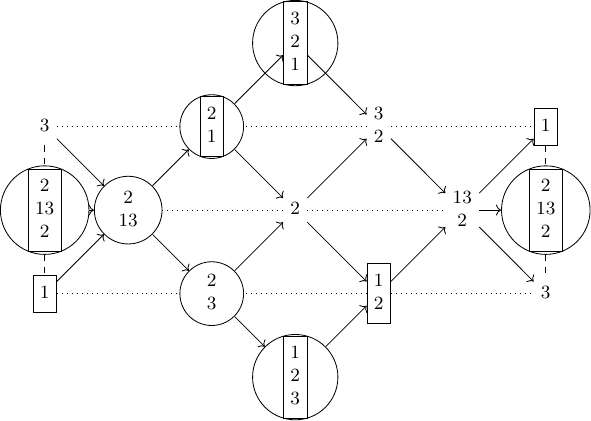}
   \end{center}
   The $\Ext^2$-quiver of $\add M$ has one vertex and one arrow and the
   $\Ext^2$-quiver of $\add N$ is empty since $N$ has no simple summands. Thus,
   this example shows that in general
   \begin{itemize}
   \item the $\Ext^d$-quiver of $\mathcal{M}$ depends on $\mathcal{M}$ and not
     only on $A$, and that
   \item the $\Ext^d$-quiver of $\mathcal{M}$ may be empty.
   \end{itemize}
   In particular, the $\Ext^d$-quiver of $\mathcal{M}$ in general does not
   determine $\mathcal{M}$ nor $A$.
 \end{example}

 \subsection{Global dimension of the higher Nakayama algebras}

 A further example shows that even for the higher Nakayama algebras, where the
 simple modules in many ways behave like simple modules in the classical sense,
 e.g. in the sense that they all lie on a unique $\tau_d$-orbit (see
 \cite{JK16}), in some respect they are different. The subsequent example shows
 that the global dimension of one of the Nakayama algebras is not attained by a
 projective dimension of a simple in the cluster-tilting subcategory as well as
 the global dimension of $A^{(d)}_{\underline{\ell}}$ is in general not equal to
 $d\cdot \gldim A^{(1)}_{\underline{\ell}}$. Moreover, the example also shows
 that, in general,
 \[
   \dim\Ext_{A^{(1)}_{\underline{\ell}}}^{2}(S_i,S_j)\neq\dim\Ext_{A^{(d)}_{\underline{\ell}}}^{2d}(S_{(i,i,\cdots,
     i)},S_{(j,j,\cdots,j)}).
 \]
 This suggest that there is no straightforward way to generalise relations in
 $A_{\underline{\ell}}^{(1)}$, which are given by the
 $\Ext_{A_{\underline{\ell}}^{(1)}}^{2}$ between simples, to higher values of
 $d$.

\begin{example}
  Let $\underline{\ell}=(1,2,2,3,3,4,3)$. The following figure shows the module
  category of $A^{(1)}_{\underline{\ell}}$. The numbers on the vertices show the
  projective dimensions of the corresponding indecomposable modules.
  \begin{center}
    \includegraphics[width=\textwidth]{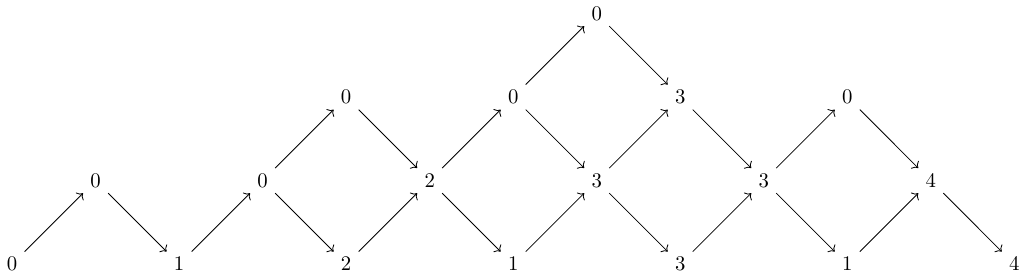}
  \end{center}
  In particular, the global dimension is $4$ and is attained by the simple
  module $S_{6}$. The following figure shows the $2$-cluster-tilting module of
  $A_{\underline{\ell}}^{(2)}$ whose indecomposable direct summands are the
  interval modules. The numbers on the vertices show the projective dimensions
  divided by $d=2$ of the corresponding indecomposable modules.
  \begin{center}
    \includegraphics[width=\textwidth]{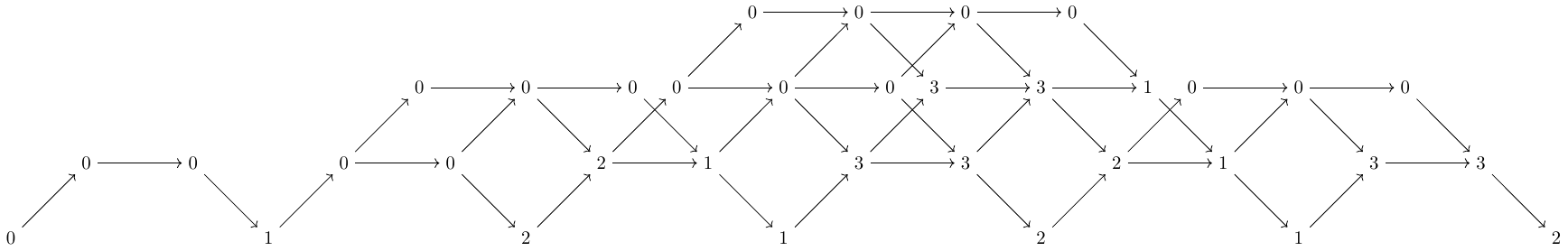}
  \end{center}
  Note that the global dimension of the algebra is $6$, but is not attained by a
  simple module in the cluster tilting subcategory. Also note that
  \[
    \gldim A^{(2)}_{\underline{\ell}}=6\neq 2\cdot 4=2\cdot \gldim
    A^{(1)}_{\underline{\ell}}.
  \]
  Finally, note that
  \[
    0\neq\dim\Ext_{A^{(1)}_{\underline{\ell}}}^{2}(S_6,S_3)\neq\dim\Ext_{A^{(2)}_{\underline{\ell}}}^{4}(S_{(6,6)},S_{(3,3)})=0
  \]
  where the rightmost vanishing condition can be shown by noting that the fifth
  term in a minimal projective resolution of $S_{(6,6)}$ does not have a summand
  $S_{(3,3)}$ in its top.
  
  It is not even true that finiteness of global dimension is preserved when
  passing to the higher analogue\footnote{The authors want to thank Yiping Chen
    for posing this question.}. To show this, consider the same example as
  before but with the last and the first vertex identified, i.e. the higher
  Nakayama algebra of type $\tilde{\mathbb{A}}_5$ with Kupisch series
  $(2,3,3,4,3,2)$. Then, $\projdim S_0=\infty$ for
  $\tilde{A}_{{\underline{\ell}}}^{(1)}$, but $\gldim
  \tilde{A}_{\underline{\ell}}^{(2)}=6<\infty$.
\end{example}

\subsection{Dominant dimension of the higher Nakayama algebras}

Concerning the dominant dimension, the authors proved in \cite{JK16} that
$\domdim A\geq d$ for every $d$-Nakayama algebra $A$. It is however also not
true that $\domdim \tilde{A}_{\underline{\ell}}^{(d)}=d\cdot \domdim
\tilde{A}_{\underline{\ell}}^{(1)}$ as the example of the higher Nakayama
algebra $\tilde{A}_{(3,4,4)}^{(d)}$ shows. For $d=1$ its dominant dimension is
$4$ as $P_{1}=I_{1}$ and $P_{2}=I_{2}$ are projective-injective and the
non-injective module $P_{0}$ has an injective coresolution
\[0\to P_{0}\to I_{1}\to I_{2}\to I_{2}\to I_{0}\to 0.\] On the other hand,
$\tilde{A}^{(2)}_{(3,4,4)}$ has $11$ indecomposable projective of which $8$ are
projective-injective. We label them by their label as interval modules for the
category $A_\infty^{(d)}$. The projective-injectives are $P_{(1,0)}=I_{(3,1)}$,
$P_{(2,1)}=I_{(4,2)}$, $P_{(2,2)}=I_{(4,3)}$, $P_{(1,1)}=I_{(4,1)}$,
$P_{(4,2)}=I_{(2,1)}$, $P_{(2,0)}=I_{(3,2)}$, $P_{(4,3)}=I_{(2,2)}$, and
$P_{(4,1)}=I_{(1,1)}$. The non-injective modules $P_{(0,0)}$, $P_{(3,1)}$, and
$P_{(3,2)}$ have the following injective coresolutions, respectively.
\begin{align*}
  &0\to P_{(3,1)}\to I_{(1,1)}\to I_{(4,1)}\to I_{(4,2)}\to I_{(4,3)}\to I_{(2,2)}\to I_{(3,2)}\to I_{(0,0)}\to 0\\
  &0\to P_{(3,2)}\to I_{(2,1)}\to I_{(4,1)}\to I_{(1,0)}\to I_{(2,0)}\to I_{(2,2)}\to I_{(3,2)}\to I_{(0,0)}\to 0\\
  &0\to P_{(0,0)}\to I_{(3,1)}\to I_{(4,1)}\to I_{(1,1)}\to I_{(2,1)}\to I_{(2,2)}\to I_{(4,3)}\to I_{(2,0)}\to 0
\end{align*}
Hence, $\domdim \tilde{A}_{(3,4,4)}^{(2)}=6\neq 2\cdot 4=2\cdot \domdim
\tilde{A}_{(3,4,4)}^{(1)}$.

\section{Cluster categories of type $\mathbb{A}_n$ and $\mathbb{A}_\infty$}
\label{sec:cluster}

In this section, we survey on the relationship between cluster tilting
subcategories and cluster categories, especially the connection to
triangulations of polygons and generalisations thereof. Already in the first
paper on cluster algebras \cite{FZ02}, Fomin and Zelevinsky studied an example
of a cluster algebra of type $\mathbb{A}_n$, the coordinate ring of the affine
cone over the Grassmannian of $2$-dimensional subspaces of an $(n+3)$-dimensional
(complex) vector space. The clusters of the cluster algebra are in bijection
with triangulations of the polygon with $n$-vertices. The flip, that is the
operation on triangulations replacing one diagonal by the unique other one
yielding a triangulation, corresponds to mutating a cluster variable. Cluster
algebras were later categorified in \cite{BMRRT06}, and the authors of that
paper realised that clusters of the cluster algebra of type $\mathbb{A}_n$ correspond to
cluster-tilting objects in the corresponding cluster category of type $\mathbb{A}_n$,
which is a certain $2$-Calabi--Yau category. As certain tilting objects in
$D^b(\mmod A_n^{(1)})$ provide cluster-tilting objects in the cluster category
of type $\mathbb{A}_n$, the concept of cluster-tilting objects can be seen as an
extension of the concept of tilting objects. There have been two generalisations
of this picture. Staying in the `$2$-dimensional' situation, the situation has
been generalised to $n=\infty$ by Holm and J\o rgensen in \cite{HJ12}
categorifying triangulations of the apeirogon, in that paper called infinitygon.
The corresponding cluster category also arises as the perfect derived category
of the cochain differential graded algebra associated to the $2$-sphere. The
other generalisation was to the higher setting, replacing the algebra
$A_n^{(1)}$ by the higher Auslander algebras $A_n^{(d)}$. This generalisation is
due to Oppermann and Thomas in \cite{OT12} and they prove that a corresponding
cluster category categorifies the triangulations of the cyclic polytope in
dimension $2d$. In a forthcoming paper \cite{JK17}, the authors will unify the
two generalisations. In this section, we survey on the results of \cite{HJ12,
  OT12, JK17}. In the first subsection, we introduce higher analogues of the
cluster categories of type $\mathbb{A}_n$. The second subsection gives the
correspondence between tilting modules for $A_n^{(d)}$ and triangulations of a
cyclic polytope, following \cite{OT12}. The third subsection is concerned with
the results of \cite{OT12} and \cite{JK17} providing a bijection between the
cluster-tilting subcategories of a cluster category and triangulations of a
cyclic polytope and a corresponding infinite version for the cyclic apeirotope.

\subsection{Categorical construction}
\label{sec:categorical_construction}

We start this section by recalling the definitions of a Calabi--Yau triangulated
category and of a spherical object.

\begin{definition}
Let $\mathcal{D}$ be a $\mathbbm{k}$-linear Hom-finite triangulated category.
  \begin{enumerate}
  \item  A functor $\mathbb{S}\colon \mathcal{D}\to \mathcal{D}$ is
    called a \emph{Serre functor} if there is a bifunctorial isomorphism
    \[D\Hom_\mathcal{D}(X,Y)\cong \Hom_\mathcal{D}(Y,\mathbb{S}X).\]
  \item The category $\mathcal{D}$ is said to be \emph{$\delta$-Calabi--Yau} if
    $[\delta]$, the $\delta$-th power of the shift functor, is a Serre functor.
  \item Let $\mathcal{D}$ be a $\delta$-Calabi--Yau triangulated category. An
    object $X\in \mathcal{D}$ is called \emph{$\delta$-spherical} if
    \[\Hom_{\mathcal{D}}(X,X[i])=\begin{cases}\mathbbm{k}&\text{if
          $i=0,\delta$},\\0&\text{else.}\end{cases}\]
  \end{enumerate}
\end{definition}

The thick subcategory generated by a $\delta$-spherical object, i.e. the smallest triangulated category containing the $\delta$-spherical object and being closed under direct summands is independent of the specific $\delta$-spherical object (especially of the ambient triangulated category), see e.g. \cite[Theorem 2.1]{KYZ09}. This
category arises in topology as follows: Let $m\geq 2$ be an integer. Let $X=S^m$
be the $m$-sphere. Let $C^*(X,\mathbbm{k})$ be the cochain differential graded
algebra of $X$ with coefficients in $\mathbbm{k}$ (with respect to the cup
product. Let $D^c(C^*(X,\mathbbm{k}))$ be the perfect derived
category of $C^*(X,\mathbbm{k})$, i.e. the subcategory spanned by all compact
objects in the derived category of the differential graded algebra
$C^*(X,\mathbbm{k})$. The following result is partially due to J\o rgensen and
partially folklore.

\begin{theorem}[cf. {\cite{Jor04}}]
  The category $D^c(C^*(X,\mathbbm{k}))$ is an $m$-Calabi--Yau category. In
  particular, it has almost split triangles. Its Auslander--Reiten quiver
  consists of $(m-1)$ components of type $\mathbb{Z}\mathbb{A}_\infty$.
  Furthermore, $D^c(C^*(X,\mathbbm{k}))$ is classically generated by an
  $m$-spherical object $S$, i.e. $D^c(C^*(X,\mathbbm{k}))$ itself is the
  smallest triangulated subcategory of $D^c(C^*(X,\mathbbm{k}))$ containing $S$.
\end{theorem}

There are several other incarnations of this category throughout mathematics.
Almost by definition, $D^c(C^*(X,\mathbbm{k}))\cong
D^c(\mathbbm{k}[\varepsilon]/(\varepsilon^2))$ where
$\mathbbm{k}[\varepsilon]/(\varepsilon^2)$ is regarded as a differential graded
algebra with trivial differential and $\varepsilon$ in degree $m$. Using Koszul
duality, it is also equivalent to the category $D^f(\mathbbm{k}[T])$, the
derived category of complexes with finite dimensional (total) cohomology over
the differential graded algebra $\mathbbm{k}[T]$ with trivial differential and
$T$ in degree $-m+1$.

In representation theory, more precisely in the theory of cluster categories,
the category is better known as the $m$-cluster category of type $\mathbb{A}_\infty$
and can be defined as $D^b(\mmod A_\infty^{(1)})/\mathbb{S}[-m]$ where
$\mathbb{S}$ denotes the Serre functor of $D^b(\mmod A_\infty^{(1)})$ which
exists because $D^b(\mmod A_\infty^{(1)})$ has almost-split triangles.

Our goal in \cite{JK17} was to obtain similar results starting from
$A_\infty^{(d)}$ instead of $A_\infty^{(1)}$. As explained in the introduction
to this section, the construction of a higher cluster category of type
$\mathbb{A}_n$ was achieved by Oppermann--Thomas in \cite{OT12}. Our proof
follows a similar strategy but additional technical complications arise. The
first step, similar to \cite{OT12} is to realise that $D^b(\mmod
A_\infty^{(d)})/\mathbb{S}[-md]$ is not triangulated. But, thanks to a
construction by Keller \cite{Kel05}, it is possible to embed it into its
triangulated hull which we denote by $\mathcal{C}_{\infty,m}^{(d)}$. The
technically difficult part is to prove that this category is in fact
$\Hom$-finite. This is achieved by establishing a derived equivalence $D^b(\mmod
A_\infty^{(d)})\cong D^b(\mmod Z_\infty^{(d)})$ where $Z_\infty^{(d)}$ is a
locally bounded category. Recall that a subcategory $\mathcal{O}\subseteq \mathcal{D}$ of a triangulated category is called \emph{weakly $d\mathbb{Z}$-cluster-tilting} if it satisfies the analogous $\Ext$-vanishing property to the case of an abelian category, but is not necessarily functorially finite (nor generating-cogenerating) and furthermore it is closed under shift by $[d]$. 

Then, a straightforward generalisation of the methods
in \cite{Ami09, Ami11, Guo11} shows the following theorem:

\begin{theorem}[{\cite[Theorems 5.14 and 5.25]{OT12} for finite $n$ and $m=2$, \cite{JK17}}]
  Let $n$ be finite or infinity. Let $ A_n^{(d)}$ be the algebra (or category)
  defined in Section \ref{sec:higher_nakayama}. Let $\mathcal{D}$ be the
  category $D^b(\mmod A_n^{(d)})$. Let $\mathbb{S}$ be the corresponding Serre
  functor on $\mathcal{D}$. The triangulated hull $\mathcal{C}_{n,m}^{(d)}$ of
  $\mathcal{D}/\mathbb{S}[-md]$ is triangulated $\Hom$-finite $md$-Calabi--Yau.

  Let $\mathcal{U}$ be the subcategory $\mathcal{U}(A_n^{(d)})$ of $D^b(\mmod
  A_n^{(d)})$. The subcategory
  $\mathcal{O}_{n,m}^{(d)}:=\mathcal{U}/\mathbb{S}[-md]$ is weakly
  $d\mathbb{Z}$-cluster-tilting in $\mathcal{C}_{\infty,m}^{(d)}$. For finite
  $n$ the `weakly' can be omitted. For infinite $n$, the category
  $\mathcal{O}_{\infty,m}^{(d)}$ is classically generated by an $md$-spherical
  object $S$, i.e. closing the subcategory $S$ under shifts and $(d+2)$-angles
  gives the whole category $\mathcal{O}_{\infty,m}^{(d)}$.
\end{theorem}

An open question is still whether it is possible to get rid of the `weakly' in
the theorem in case $n=\infty$. The problem is that it is not clear whether in
this case $D^b(\mmod A_\infty^{(d)})/\mathbb{S}[-md]$ is functorially finite in
its triangulated hull which would be a sufficient condition.

Another open question is whether for $n=\infty$ there is a natural construction
associating these categories to certain topological objects which generalises
the case of $d=1$.

\subsection{Tilting modules for $A_n^{(d)}$}

Recall that there is a bijection between the isomorphism classes of
indecomposable basic tilting $A_n^{(1)}$-modules and the triangulations of the
regular polygon with $n+2$ vertices; moreover, this bijection is compatible with
tilting mutation on one side and flip of triangulations on the other side
\cite{BK04, OT12}.

The higher dimensional analogues of regular polygons are widely considered to be
the cyclic polytopes \cite[Section 6.1]{DLRS2010}. In order to define them let
$\setP{(t,t^2,\dots,t^{\delta})}{t\in \mathbb{R}}$ be the \emph{moment curve}.
Choose $p$ points on it. Then, a \emph{cyclic polytope} is the convex hull
$C(p,\delta)$ of such points. A \emph{triangulation} of a cyclic polytope is a
subdivision of it into $\delta$-simplices whose vertices are also vertices of
$C(p,\delta)$. By a result of Dey \cite{Dey93} the analogy to triangulations of
regular polygons is particularly strong if $\delta=2d$ is even. In this case, a
triangulation is determined by the $d$-dimensional internal simplices, more
precisely:

\begin{theorem}[cf. {\cite[Theorem 2.4]{OT12}}]
  Specifying a triangulation of $C(p,2d)$ is equivalent to giving a collection
  of $p-d-1\choose d$ non-intersecting $d$-simplices which do not lie on a lower
  boundary facet with vertices in the $p$ points.
\end{theorem}

For $d=1$, every maximal collection of $1$-simplices has $p-2$ elements. For
$d>2$ this is in general no longer true, see \cite{OT12}.

There is an involutive operation on the set of triangulations of $C(p,\delta)$
called ``bistellar flip''. On the corresponding collection of non-intersecting
$d$-simplices it can be described by removing a $d$-simplex and replacing it by
the unique other $d$-simplex giving rise to a triangulation. When $\delta=2$,
this reduces to the usual flip of triangulations of planar polygons, see
\emph{loc. cit.} for further details.

The corresponding notion on the categorical side is tilting mutation: Given two
tilting modules $T$ and $T'$ contained in $\add M_n^{(d)}$ which only differ in
indecomposable direct summands $X$ and $X'$, $T$ is called the mutation of $T'$
in $X'$ and vice versa. In this case, there exists either an exact sequence
\[0\to X\to T_d\to \dots\to T_1\to X'\to 0\] with $T_i\in \add \overline{T}$
where $\overline{T}$ is a direct complement to $X$ in $T$ (resp. to $X'$ in
$T'$) or an exact sequence in the other direction.

\begin{theorem}[{\cite[Thm. 1.1]{OT12}}]
  There is a bijection
  \[\begin{tikzcd}[row sep=small]
      \set{\text{internal $d$-simplices of $C(n+2d,2d)$}}\dar[leftrightarrow]\\
      \set{\text{indecomposable non-projective-injective summands of
          $M_n^{(d)}$}}\end{tikzcd}\] which in turn induces a bijection
  \[\begin{tikzcd}[row sep=small]
      \set{\text{triangulations of $C(n+2d,2d)$}}\dar[leftrightarrow]\\
      \set{\text{basic tilting modules for $A_n^{(d)}$ contained in
          $M_n^{(d)}$}}.\end{tikzcd}\] Moreover, this bijection is compatible
  with tilting mutation on one side and bistellar flip of triangulations on the
  other side.
\end{theorem}

The second bijection follows from the fact that under the first bijection two
$d$-simplices intersect if and only if the corresponding indecomposable modules
have non-trivial $\Ext^d$ in some direction. Note that because of the property
of $M_n^{(d)}$ being cluster-tilting and the fact that $\gldim A_n^{(d)}=d$, the
space $\Ext^d$ is the only possible extension not to vanish. This theorem gives
a first categorification of the triangulations of a cyclic polytope. In the
forthcoming subsection, we give a second one using cluster tilting subcategories
of a cluster category instead of tilting modules in a module category.

\subsection{Triangulations of cyclic polytopes and apeirotopes}

In this subsection, we describe the combinatorial model for the cluster-tilting
subcategories of the category $\mathcal{O}_{n,2}^{(d)}$ for finite or infinite
$n$, introduced in Subsection \ref{sec:categorical_construction}, in more
detail. This has been done for finite $n$ by Oppermann--Thomas \cite{OT12} and
for $n=\infty$ by the authors in \cite{JK17} building on earlier work
\cite{FZ02, BMRRT06, HJ12}. Following \cite{OT12}, we need to replace
functorially finiteness in the definition of cluster-tilting by a stronger
condition in the category $\mathcal{O}_{n,2}^{(d)}$:

\begin{definition}
  An object $T\in \mathcal{O}_{n,2}^{(d)}$ is called \emph{cluster-tilting} if
  \begin{enumerate}
  \item $\Hom_{\mathcal{O}_{n,2}^{(d)}}(T,T[d])=0$, and
  \item Any $X\in \mathcal{O}_{n,2}^{(d)}$ occurs in a $(d+2)$-angle
    \[X[-d]\to T_d\to T_{d-1}\to \dots\to T_1\to T_0\to X\] with $T_i\in \add
    T$.
  \end{enumerate}
\end{definition}

With this definition, for finite $n$, Oppermann and Thomas were able to prove
that cluster-tilting subcategories of $\mathcal{O}_{n,2}^{(d)}$ are precisely
the cluster-tilting subcategories of $\mathcal{C}_{n,2}^{(d)}$ contained in
$\mathcal{O}_{n,2}^{(d)}$. For infinite $n$ the problem is again that it is not
clear whether $\mathcal{O}_{\infty,2}^{(d)}$ is functorially finite in
$\mathcal{C}_{\infty,2}^{(d)}$. The only ingridient missing before stating
Oppermann--Thomas' categorification of triangulations of the cyclic polytope by
cluster-tilting subcategories of the higher cluster category of type
$\mathbb{A}$ is the notion of mutation of cluster-tilting objects.

\begin{definition}
  Let $X$ and $Y$ be indecomposable objects of $\mathcal{O}_{n,2}^{(d)}$ and
  $\overline{T}$ be such that $X\oplus \overline{T}$ and $Y\oplus \overline{T}$
  are cluster-tilting objects. Then $X\oplus \overline{T}$ is called a mutation
  of $Y\oplus \overline{T}$ at $Y$ and vice versa.
\end{definition}

In this case, there are \emph{exchange $(d+2)$-angles}
\[X\to E_d\to \dots\to E_1\to Y\] and
\[Y\to F_1\to \dots\to F_d\to X\] with $E_i, F_j\in \add \overline{T}$.

\begin{theorem}[{\cite[Thm. 1.2]{OT12}}]
  There is a bijection
  \[\begin{tikzcd}[row sep=small]
      \set{\text{internal $d$-simplices of $C(n+2d+1,2d)$}}\dar[leftrightarrow]\\
      \set{\text{indecomposable objects in
          $\mathcal{O}_{n,2}^{(d)}$}}\end{tikzcd}
  \]
  which in turn induces a bijection
  \[\begin{tikzcd}[row sep=small]
      \set{\text{triangulations of $C(n+2d+1,2d)$}}\dar[leftrightarrow]\\
      \set{\text{basic $2d$-cluster-tilting objects in
          $\mathcal{O}_{n,2}^{(d)}$}}.\end{tikzcd}
  \]
  The bistellar flip on the left hand side corresponds to the mutation of
  cluster-tilting objects on the right hand side.
\end{theorem}

For our generalisation to the infinite situation, we consider the \emph{cyclic
  apeirotope}, i.e. the convex hull $C(\infty,2d)$ of
$\setP{(t,t^2,\dots,t^{2d})}{t\in \mathbb{Z}}$. An \emph{Ind-finite
  triangulation} is a collection of internal $d$-simplices such that for all
$I\subseteq \mathbb{Z}$ finite, there exists an interval $[a,b]$ containing $I$
such that the arcs with end points in $\setP{(t,t^2,\dots,t^{2d})}{t\in
  \mathbb{Z}\cap [a,b]}$ form the internal $d$-simplices of a triangulation.

\begin{theorem}
  There is a bijection
  \[\begin{tikzcd}[row sep=small]
      \set{\text{internal $d$-simplices of $C(\infty,2d)$}}\dar[leftrightarrow]\\
      \set{\text{indecomposable objects in
          $\mathcal{O}_{\infty,2}^{(d)}$}}\end{tikzcd}
  \]
  which in turn induces a bijection
  \[\begin{tikzcd}[row sep=small]
      \set{\text{Ind-finite triangulations of
          $C(\infty,2d)$}}\dar[leftrightarrow]\\ \set{\text{basic
          $2d$-cluster-tilting subcategories in
          $\mathcal{O}_{\infty,2}^{(d)}$}}.\end{tikzcd}
  \]

  The bistellar flip on the left hand side again corresponds to the mutation of
  $d$-cluster-tilting objects on the right hand side.
\end{theorem}

\bibliographystyle{alpha}

\bibliography{mathscinet}

\newcommand{\etalchar}[1]{$^{#1}$}
\begin{thebibliography}{MVdlP83}

\bibitem[Ami09]{Ami09}
Claire Amiot.
\newblock Cluster categories for algebras of global dimension 2 and quivers
  with potential.
\newblock {\em Ann. Inst. Fourier (Grenoble)}, 59(6):2525--2590, 2009.

\bibitem[Ami11]{Ami11}
Claire Amiot.
\newblock On generalized cluster categories.
\newblock In {\em Representations of algebras and related topics}, EMS Ser.
  Congr. Rep., pages 1--53. Eur. Math. Soc., Z\"urich, 2011.

\bibitem[ASS06]{ASS06}
Ibrahim Assem, Daniel Simson, and Andrzej Skowro\'nski.
\newblock {\em Elements of the representation theory of associative algebras.
  {V}ol. 1}, volume~65 of {\em London Mathematical Society Student Texts}.
\newblock Cambridge University Press, Cambridge, 2006.
\newblock Techniques of representation theory.

\bibitem[Aus71]{Aus71}
Maurice Auslander.
\newblock Representation dimension of {A}rtin algebras.
\newblock Lecture Notes, Queen Mary College, London, 1971.

\bibitem[BK04]{BK04}
Aslak~Bakke Buan and Henning Krause.
\newblock Tilting and cotilting for quivers of type {$\tilde{A}_n$}.
\newblock {\em Journal of Pure and Applied Algebra}, 190(1--3):1--21, 2004.

\bibitem[BMR{\etalchar{+}}06]{BMRRT06}
Aslak~Bakke Buan, Robert Marsh, Markus Reineke, Idun Reiten, and Gordana
  Todorov.
\newblock Tilting theory and cluster combinatorics.
\newblock {\em Advances in Mathematics}, 204(2):572--618, 2006.

\bibitem[Dey93]{Dey93}
Tamal~Krishna Dey.
\newblock On counting triangulations in {$d$} dimensions.
\newblock {\em Computational Geometry. Theory and Applications}, 3(6):315--325,
  1993.

\bibitem[DLRS10]{DLRS2010}
Jes\'us~A. De~Loera, J\"org Rambau, and Francisco Santos.
\newblock {\em Triangulations}, volume~25 of {\em Algorithms and Computation in
  Mathematics}.
\newblock Springer-Verlag, Berlin, 2010.
\newblock Structures for algorithms and applications.

\bibitem[FZ02]{FZ02}
Sergey Fomin and Andrei Zelevinsky.
\newblock Cluster algebras. {I}. {F}oundations.
\newblock {\em Journal of the American Mathematical Society}, 15(2):497--529,
  2002.

\bibitem[Gre83]{Gre83}
Edward~L. Green.
\newblock Grrelation relations, coverings and group-graded algebras.
\newblock {\em Transacations of the American Mathematical Society},
  279(1):297--310, 1983.

\bibitem[Guo11]{Guo11}
Lingyan Guo.
\newblock Cluster tilting objects in generalized higher cluster categories.
\newblock {\em J. Pure Appl. Algebra}, 215(9):2055--2071, 2011.

\bibitem[HI11a]{HI11}
Martin Herschend and Osamu Iyama.
\newblock {$n$}-representation-finite algebras and twisted fractionally
  {C}alabi-{Y}au algebras.
\newblock {\em Bulletin of the London Mathematical Society}, 43(3):449--466,
  2011.

\bibitem[HI11b]{HI11b}
Martin Herschend and Osamu Iyama.
\newblock Selfinjective quivers with potential and 2-representation-finite
  algebras.
\newblock {\em Compositio Mathematica}, 147(6):1885--1920, 2011.

\bibitem[HJ12]{HJ12}
Thorsten Holm and Peter J{\o}rgensen.
\newblock On a cluster category of infinite {D}ynkin type, and the relation to
  triangulations of the infinity-gon.
\newblock {\em Mathematische Zeitschrift}, 270(1--2):277--295, 2012.

\bibitem[IJ16]{IJ16}
Osamu Iyama and Gustavo Jasso.
\newblock Higher {A}uslander correspondence for dualizing {$R$}-varieties.
\newblock {\em Algebras and Representation Theory}, pages 1--20, 2016.

\bibitem[IO11]{IO11}
Osamu Iyama and Steffen Oppermann.
\newblock $n$-representation-finite algebras and $n$-{APR} tilting.
\newblock {\em Transactions of the American Mathematical Society},
  363(12):6575--6614, July 2011.

\bibitem[IO13]{IO13}
Osamu Iyama and Steffen Oppermann.
\newblock Stable categories of higher preprojective algebras.
\newblock {\em Advances in Mathematics}, 244:23--68, 2013.

\bibitem[IS16]{IS16}
Osamu Iyama and {\O}yvind Solberg.
\newblock Auslander-{G}orenstein algebras and precluster tilting.
\newblock preprint, 2016.

\bibitem[Iya07a]{Iya07}
Osamu Iyama.
\newblock Auslander correspondence.
\newblock {\em Advances in Mathematics}, 210(1):51--82, 2007.

\bibitem[Iya07b]{Iya07b}
Osamu Iyama.
\newblock Higher-dimensional {A}uslander-{R}eiten theory on maximal orthogonal
  subcategories.
\newblock {\em Advances in Mathematics}, 210(1):22--50, 2007.

\bibitem[Iya08]{Iya08}
Osamu Iyama.
\newblock Auslander-{R}eiten theory revisited.
\newblock In {\em Trends in representation theory of algebras and related
  topics}, EMS Series of Congress Reports, pages 349--397. European
  Mathematical Society, Z\"urich, 2008.

\bibitem[Iya11]{Iya11}
Osamu Iyama.
\newblock Cluster tilting for higher {A}uslander algebras.
\newblock {\em Advances in Mathematics}, 226(1):1--61, 2011.

\bibitem[Jas16]{Jas16}
Gustavo Jasso.
\newblock {$n$}-abelian and {$n$}-exact categories.
\newblock {\em Math. Z.}, 283(3-4):703--759, 2016.

\bibitem[JK16a]{JK16}
Gustavo Jasso and Julian K\"ulshammer.
\newblock Higher {N}akayama algebras {I}: Construction.
\newblock preprint, arXiv: 1604.03500, 2016.

\bibitem[JK16b]{JK16b}
Gustavo Jasso and Julian K\"ulshammer.
\newblock The naive approach for constructing the derived category of a
  $d$-abelian category fails.
\newblock note, arXiv: 1604.03473, 2016.

\bibitem[JK16c]{JK16c}
Gustavo Jasso and Sondre Kvamme.
\newblock An introduction to higher {A}uslander--{R}eiten theory.
\newblock preprint, 2016.

\bibitem[JK17]{JK17}
Gustavo Jasso and Julian K\"ulshammer.
\newblock Spherical objects and cluster-tilting.
\newblock in preparation, 2017.

\bibitem[J{\o}r04]{Jor04}
Peter J{\o}rgensen.
\newblock Auslander--{R}eiten theory over topological spaces.
\newblock {\em Comment. Math. Helv.}, 79(1):160--182, 2004.

\bibitem[Kel05]{Kel05}
Bernhard Keller.
\newblock On triangulated orbit categories.
\newblock {\em Doc. Math.}, 10:551--581, 2005.

\bibitem[KYZ09]{KYZ09}
Bernhard Keller, Dong Yang, and Guodong Zhou.
\newblock The {H}all algebra of a spherical object.
\newblock {\em Journal of the London Mathematical Society. Second Series},
  80(3):771--784, 2009.

\bibitem[Mar17]{Mar17}
Rene Marczinzik.
\newblock Finitistic {A}uslander algebras.
\newblock preprint, 2017.

\bibitem[MOS09]{MOS09}
Volodymyr Mazorchuk, Serge Ovsienko, and Catharina Stroppel.
\newblock Quadratic duals, {K}oszul dual functors, and applications.
\newblock {\em Transacations of the American Mathematical Society},
  361(3):1129--1172, 2009.

\bibitem[MVdlP83]{MVdlP83}
Roberto Mart{\'\i}nez-Villa and Jos{\'e}~Antonio de~la Pe{\~n}a.
\newblock The universal cover of a quiver with relations.
\newblock {\em Journal of Pure and Applied Algebra}, 30(3):277--292, 1983.

\bibitem[OT12]{OT12}
Steffen Oppermann and Hugh Thomas.
\newblock Higher-dimensional cluster combinatorics and representation theory.
\newblock {\em J. Eur. Math. Soc. (JEMS)}, 14(6):1679--1737, 2012.

\bibitem[SS07a]{SS07a}
Daniel Simson and Andrzej Skowro\'nski.
\newblock {\em Elements of the representation theory of associative algebras.
  {V}ol. 2}, volume~71 of {\em London Mathematical Society Student Texts}.
\newblock Cambridge University Press, Cambridge, 2007.
\newblock Tubes and concealed algebras of Euclidean type.

\bibitem[SS07b]{SS07b}
Daniel Simson and Andrzej Skowro\'nski.
\newblock {\em Elements of the representation theory of associative algebras.
  {V}ol. 3}, volume~72 of {\em London Mathematical Society Student Texts}.
\newblock Cambridge University Press, Cambridge, 2007.
\newblock Representation-infinite tilted algebras.

\end{thebibliography}

\end{document}